\documentclass[12pt]{article}
\usepackage{amsmath, amsthm, amssymb, epsfig, graphicx, fullpage}

\pdfoutput=1

\newtheorem{theorem}{Theorem}[section]
\newtheorem{lemma}[theorem]{Lemma}

\newtheorem{proposition}[theorem]{Proposition}
\newtheorem{corollary}[theorem]{Corollary}

\newcommand{\haussdim}[0]{\mathrm{dim_H} \,}

\newcommand{\prob}[0]{\mathrm{\textbf{P}}}
\newcommand{\expect}[0]{\mathrm{\textbf{E}}}
\newcommand{\indicate}[1]{\mathbf{1} \left \{ #1 \right \}}
\newcommand{\dist}{\operatorname{dist}}

\begin{document}

\title{Hausdorff Dimension of the SLE Curve Intersected with the Real Line}
\author{\begin{tabular}{cc}
Tom Alberts & Scott Sheffield\footnote{Research supported in part by NSF Grants DMS 0403182 and DMS 064558.} \\
\small{E-mail: \texttt{alberts@cims.nyu.edu}} &
\small{E-mail: \texttt{sheff@cims.nyu.edu}} \\
\end{tabular} \\
\\
\small{Courant Institute of Mathematical Sciences} \\
\small{251 Mercer Street} \\
\small{New York, NY 10012} \\
\\}
\date{}

\maketitle

\begin{abstract}
We establish an upper bound on the asymptotic probability of an
$\textrm{SLE}(\kappa)$ curve hitting two small intervals on the real
line as the interval width goes to zero, for the range $4 < \kappa <
8$. As a consequence we are able to prove that the random set of
points in $\mathbb{R}$ hit by the curve has Hausdorff dimension
$2-8/\kappa$, almost surely.
\end{abstract}



\section{Introduction \label{Intro}}

In the seminal paper \cite{rohde_schramm}, Rohde and Schramm were
able to prove that the Hausdorff dimension of an SLE$(\kappa)$ curve
is almost surely less than or equal to $\min(1 + \kappa/8, 2)$. The
scaling properties of SLE immediately imply that the Hausdorff
dimension of the curve must almost surely be a constant, and they
conjectured that their bound was in fact sharp. In general though,
proving a sharp lower bound on the dimension of a random set is a
difficult task. In \cite{lawler:bolyai}, Lawler describes a widely
applicable and commonly used method for doing so. The required
ingredient is a very precise estimate on the probability of two
balls both intersecting the random set. Often this is referred to as
a second moment method since it can be used to get bounds on the
variance of the number of balls (of a certain radius) needed to
cover the set. The second moment estimate is difficult as it has to
precisely describe how the probability decays as the radius of the
balls shrink to zero, and as the balls move closer and farther
apart. In the case of the SLE curve, Beffara was able to establish
the necessary second moment estimates in \cite{beffara:curvedim}.
Lawler \cite{lawler:curvedim} has recently announced a new proof of
the lower bound by using a modified version of the second moment
method that does not explicitly require an estimate on the two-ball
hitting probability.

In this paper we prove a result on the almost sure Hausdorff
dimension of another random set arising from the Schramm-Loewner
Evolution, namely the set of points at which the curve intersects
the real line. Let $\gamma$ be a chordal SLE$(\kappa)$ curve from
zero to infinity in the upper half plane $\mathbb{H}$ of
$\mathbb{C}$. The interaction of this curve with the real line
depends very strongly on the well-known phase transitions of SLE. In
the case $0 \leq \kappa \leq 4$ the curve is almost surely simple
and intersects $\mathbb{R}$ only at zero. For $\kappa \geq 8$ the
curve is space-filling and so $\gamma[0, \infty) \cap \mathbb{R} =
\mathbb{R}$. For the purposes of this paper the most interesting
range is $4 < \kappa < 8$, in which the curve intersects
$\mathbb{R}$ on a random Cantor-like set of Hausdorff dimension less
than 1. The fractal nature of $\gamma[0, \infty) \cap \mathbb{R}$
should not be surprising. When the curve does hit the real line it
tends to linger for a while and hit other real points before
wandering off into the upper half plane again, which gives the set
of hit points enough irregularity to have a fractional dimension.
The main result of this paper is the following:

\begin{theorem}
\label{MainTheorem} For $4 < \kappa < 8$, the Hausdorff dimension of
the set $\gamma[0, \infty) \cap \mathbb{R}$ is almost surely $2 -
8/\kappa$.
\end{theorem}

It is worth noting that the dimension in Theorem \ref{MainTheorem}
is the unique affine function of $1/\kappa$ that interpolates
between the already known dimension values of $0$ for $\kappa \leq
4$, and $1$ for $\kappa \geq 8$. In contrast, the Hausdorff
dimension of the SLE$(\kappa)$ curve itself is an affine function of
$\kappa$ for $0 \leq \kappa \leq 8$.

We will prove Theorem \ref{MainTheorem} using the second moment
method described in \cite{lawler:bolyai}. The asymptotics of certain
hitting probabilities, already well established in a number of
papers (see Section \ref{UpperBoundSection}), give the upper bound
on the dimension. New results of this paper, which establish the
asymptotics of the SLE curve hitting two disjoint small intervals on
the real line, give the lower bound.

An alternative (and independently obtained) proof of Theorem
\ref{MainTheorem} was announced by Schramm and Zhou in
\cite{schramm_zhou:dimension}. Schramm and Zhou do not obtain
explicit bounds on the probability that the path hits two disjoint
intervals (as we do here).  Rather, instead of working with
$\gamma[0, \infty) \cap \mathbb{R}$ directly, they use an explicit
martingale to construct a measure (a so-called {\em Frostman
measure}) on a particular subset of $\gamma[0, \infty) \cap
\mathbb{R}$, which allows them to bound the Hausdorff dimension of
both sets from below.

\subsection{Preliminaries \label{prelim}}

In this paper we work exclusively with the chordal form of Loewner's
equation in the upper half plane. Given a continuous, real-valued
function $t \mapsto U_t, t \geq 0$, the map $g_t(z)$ is defined to
be the unique solution to the initial value problem
\begin{align*}
\partial_t g_t(z) = \frac{2}{g_t(z) - U_t}, \,\, g_0(z) = z.
\end{align*}
An important feature of the maps $g_t$ is that they satisfy the
hydrodynamic normalization at infinity, i.e. $g_t(z) = z + o(1)$ as
$z \to \infty$. Schramm-Loewner Evolution, or more precisely chordal
SLE$(\kappa)$ from 0 to infinity in $\mathbb{H}$, corresponds to the
choice $U_t = \sqrt{\kappa} B_t$, where $B_t$ is a standard
1-dimensional Brownian motion. The results of this paper hold
exclusively for SLE$(\kappa)$, but many of the Lemmas we derive are
deterministic in nature and hold for any continuous driving
function. To emphasize this point and keep the deterministic results
separate from the probabilistic ones we, for these Lemmas, denote
the driving function by $U_t$.

As most of the exponents in this paper usually involve terms in
$1/\kappa$ rather than $\kappa$, we have chosen to use the slightly
different SLE notation that has been championed by Lawler. Instead
of $\kappa$ he uses the parameter $a = 2/\kappa$, and the form of
the Loewner equation defined by
\begin{align}
\partial_t g_t(z) = \frac{a}{g_t(z) - B_t}, \,\, g_0(z) = z.
\label{SLE_def2}
\end{align}
For any $z \in \overline{\mathbb{H}}$ the function $g_t(z)$ is
well-defined up to a random time $T_z$. It is clear from
\eqref{SLE_def2} that $T_z$ is the first time $t$ at which $g_t(z) -
B_t = 0$. Let $K_t = \overline{\{ z \in \mathbb{H} : T_z \leq t \}}$
which is a compact, connected subset of $\overline{\mathbb{H}}$
called the SLE hull. In \cite{rohde_schramm} it was proven that for
all values of $\kappa$ the hull is generated by a curve $\gamma: [0,
\infty) \to \overline{\mathbb{H}}$, i.e. for all $t$, $\mathbb{H}
\backslash K_t$ is the unbounded connected component of $\mathbb{H}
\backslash \gamma([0,t])$. If $1/4 < a < 1/2$ (corresponding to $4 <
\kappa < 8$) then $K_\infty \cap \mathbb{R} = \mathbb{R}$ but
$\gamma[0, \infty) \cap \mathbb{R}$ is a proper subset of
$\mathbb{R}$. The latter fact is evident by observing that
$\gamma[0, \infty) \cap \mathbb{R}$ is determined by the process
$T_x$ for $x \in \mathbb{R}$. If $x > y > 0$ then the curve
intersects $\mathbb{R}$ between $y$ and $x$ iff $T_x > T_y$, and in
the case $1/4 < a < 1/2$ there is always a positive probability of
having $T_x = T_y$. In fact this last probability can be computed
exactly (see \cite[Propositions 6.8 \& 6.34]{lawler:book} for a
detailed discussion), and it is from the asymptotics of this
probability as $x \downarrow y$ that we obtain the upper bound on
the Hausdorff dimension.

Two well known scaling properties of SLE we will use throughout are
that $T_x$ is identical in law to $x^2T_1$, and that if $\gamma$ is
an SLE curve then $\gamma_r(t) := r^{-1} \gamma(r^2t)$ is a curve
identical in law to $\gamma$ (see, e.g., \cite{rohde_schramm}). The
latter, combined with the symmetry of the SLE process about the
imaginary axis, tells us that to compute the Hausdorff dimension of
$\gamma[0, \infty) \cap \mathbb{R}$ it is enough to consider only $
\gamma[0, \infty) \cap [0,1] = \gamma[0, T_1] \cap [0,1]$.

Scaling properties also immediately imply the following.

\begin{lemma}
\label{asDimLemma} The Hausdorff dimension of $\gamma[0,T_1] \cap
[0,1]$ is almost surely a constant.
\end{lemma}

\begin{proof}
The following argument is by now standard (see
\cite{beffara:sle6dim}, for instance). Let $A_x = \gamma[0, T_x]
\cap [0, x]$. The scaling relations tell us that $A_x$ has the same
law as $xA_1$ for all $x > 0$, and since Hausdorff dimension is
unchanged under linear scaling we have $\haussdim xA_1 = \haussdim
A_1$. Thus $\haussdim A_x$ is equal in law to $\haussdim A$ for all
$x > 0$. Now $\haussdim A_x$ is a decreasing quantity as $x
\downarrow 0$ so it converges almost surely, and its limit has the
same distribution as $\haussdim A_1$ and is
$\mathcal{F}_{0+}$-measurable (the sigma algebra is that of the
Brownian motion). By Blumenthal 0-1 Law the limit must be a
constant. Hence $\haussdim A_1$ is equal in law to a constant and
therefore a constant itself.
\end{proof}

\subsection{Method of Calculating the Hausdorff Dimension
\label{HaussMethod}}

A standard procedure for calculating the Hausdorff dimension of
random subsets of $[0,1]$ is described in \cite{lawler:bolyai}. The
main idea is to finely partition the unit interval and compute
statistics on the number of subintervals that intersect the random
subset. For integer $n \geq 1$ and $1 \leq k \leq 2^n$ define $D_k^n
= \{ T(k2^{-n}) > T((k-1)2^{-n}) \}$, which is the event that the
SLE curve hits in the interval $[(k-1)2^{-n}, k2^{-n}]$. The next
Lemma shows how to prove the upper bound on the Hausdorff dimension.

\begin{lemma}[\cite{lawler:bolyai}, Lemma 1]
\label{UpperBoundLemma} If $s \in (0,1)$ and there exists a $C <
\infty$ such that for all sufficiently large $n$,
\begin{align}
\sum_{k=1}^{2^n} \prob \left( D_k^n \right)  \leq C2^{sn},
\label{upperbound}
\end{align}
then almost surely $\haussdim \gamma[0, T_1] \cap [0,1] \leq s$.
\end{lemma}

Showing that the same $s$ is in fact a lower bound is usually a more
difficult task, and it is accomplished by establishing the following
estimates.

\begin{lemma}[\cite{lawler:bolyai}, Lemma 2]
\label{LowerBoundLemma} If $s \in (0,1)$, and there exists $C_1, C_2
\in (0, \infty)$ and $\delta \in (0,1/2)$ such that
\begin{align}
\prob \left( D_k^n \right) \geq C_1 2^{-(1-s)n}, \,\, \mathrm{for}
\,\,\, \delta \leq \frac{k}{2^n} \leq 1 - \delta,
\label{lowerbound1}
\end{align}
and
\begin{align}
\prob \left( D_j^n \cap D_k^n \right) \leq C_2
2^{-(1-s)n}(k-j)^{-(1-s)}, \,\, \mathrm{for} \,\,\, \delta \leq
\frac{j}{2^n} < \frac{k}{2^n} \leq 1-\delta, \label{lowerbound2}
\end{align}
for all $n$ sufficiently large, then there exists a $p = p(s, C_1,
C_2, \delta)
> 0$ such that
\begin{align*}
\textbf{P} \left( \haussdim \left( \gamma[0,T_1] \cap [\delta, 1 -
\delta] \right) \geq s \right) \geq p.
\end{align*}
\end{lemma}

In the present paper we take $s = 2 - 8/\kappa = 2 - 4a$. Section
\ref{UpperBoundSection} summarizes the results that give us
\eqref{upperbound}. Establishing estimates \eqref{lowerbound1} and
\eqref{lowerbound2} is the focus of Section \ref{LowerBoundSection}.
Combined with Lemma \ref{asDimLemma} these three estimates will
prove Theorem \ref{MainTheorem}.

\section{The One-Interval Estimate \label{UpperBoundSection}}

In this section we consider the probability of an SLE curve hitting
a specified interval on the positive real axis. An exact formula
exists and was first proven in \cite{rohde_schramm}. Also see
\cite[Proposition 6.34]{lawler:book} for another proof. We will make
use of a more general version proven in \cite{dubedat:triangle}.

\begin{proposition}[{\cite[Proposition 1]{dubedat:triangle}}]
\label{HittingProb} For chordal SLE($\kappa$) with $4 < \kappa < 8$,
define $F : \mathbb{H} \to T$ to be a Schwarz-Christoffel map from
$\mathbb{H}$ into an isosceles triangle $T$ that sends $0, 1,$ and
$\infty$ to the vertices, with interior angle $(4a-1) \pi$ at the
vertex $F(1)$ and equal angles at the other two vertices (see Figure
\ref{triangle-fig}). Then
\begin{align*}
F(z) = F(0) \prob(T_z < T_1) + F(1) \prob(T_z = T_1) + F(\infty)
\prob(T_z > T_1),
\end{align*}
that is, the three swallowing probabilities are the weights that
make $F(z)$ a convex combination of the three vertices $F(0), F(1),$
and $F(\infty)$.
\end{proposition}

\begin{figure}
\begin{center}
\includegraphics[scale=1]{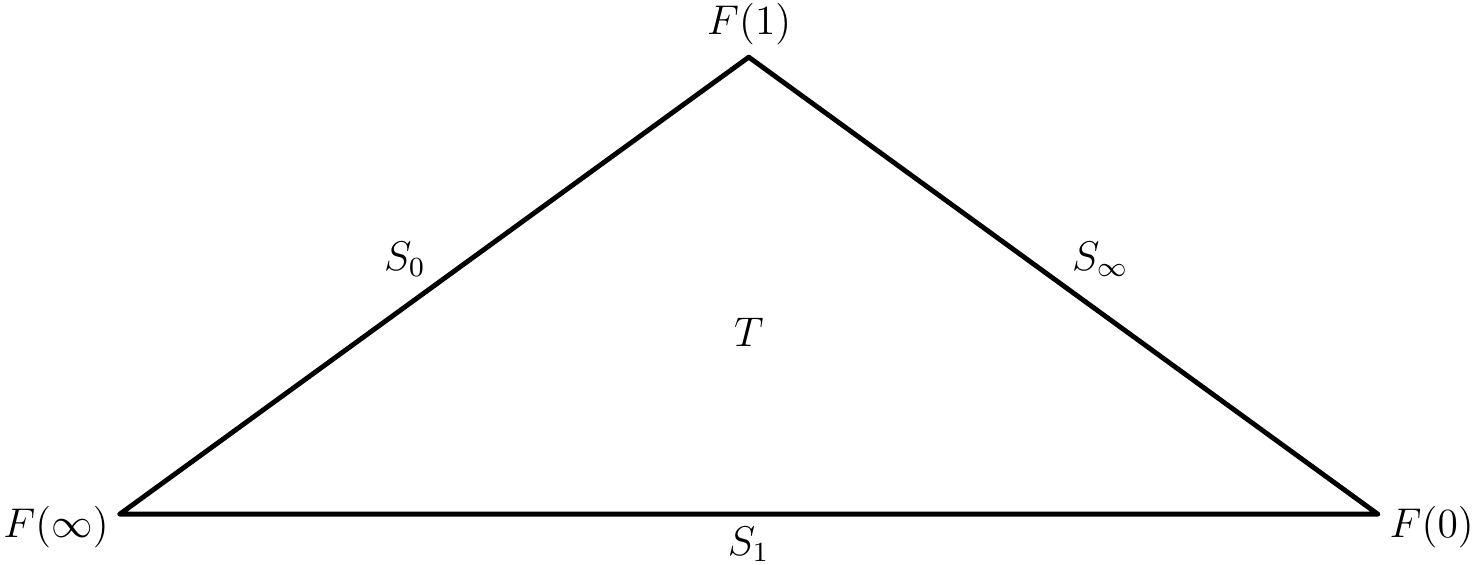}
\caption{An example of the triangle $T$ used in Proposition
\ref{HittingProb}.} \label{triangle-fig}
\end{center}
\end{figure}

The weights used in the above convex combination are commonly called
the \textit{barycentric coordinates} of the point $F(z)$ in the
triangle $T$. Up to translation, scaling, and rotation of the
triangle $T$, the map $F$ is determined by the condition $F'(z)
\propto z^{-2a}(1-z)^{4a-2}$ (here $f(z) \propto g(z)$ means $f(z) =
\zeta g(z)$ for some $\zeta \in \mathbb{C} \backslash \{0\}$). In
subsequent discussion, we will use the choice of $F$ defined by
\begin{align}
F(z) = \frac{\Gamma(2a)}{\Gamma(1-2a) \Gamma(4a-1)} \int_0^{1-z}
\frac{d\xi}{\xi^{2-4a}(1-\xi)^{2a}}. \label{ExplicitF}
\end{align}
This is the choice of $F$ for which no extra scaling or translation
is required to express the hitting probability $\prob \left( T_x <
T_y \right)$, as in the next Proposition. Note that the integral is
single-valued in $\overline{\mathbb{H}}$ with $F(1) = 0$ and $F(0) =
1$ (the integral defining $F(0)$ is a standard beta integral).

We now use Proposition \ref{HittingProb} to establish some further
results that will be useful in later computations. Here and
throughout this paper we will use the notation $f(s) \asymp g(s)$ to
mean there exists constants $0 < C_1 < C_2$ such that $C_1 f(s) \leq
g(s) \leq C_2 g(s)$, for all values of the parameter $s$.

\begin{corollary}
\label{RealLineHittingProb} If $x,y \in \mathbb{R}, x > y > 0$, then
$\prob(T_x > T_y) = F(y/x)$, and consequently
\begin{align}
\prob (T_x > T_y) \asymp \left( \frac{x-y}{x} \right)^{4a-1}.
\label{AsympExpression}
\end{align}
The constants implicit in $\asymp$ depend only on $a$. Moreover if
$\tau$ is any deterministic time or stopping time such that $\tau <
T_y$, then
\begin{align*}
\prob \left( T_x > T_y \mid \mathcal{F}_{\tau} \right) = F \left(
\frac{g_{\tau}(x) - g_{\tau}(y)}{g_{\tau}(x) - B_{\tau}} \right)
\asymp \left( \frac{g_{\tau}(x) - g_{\tau}(y)}{g_{\tau}(x) -
B_{\tau}} \right)^{4a-1}.
\end{align*}
\end{corollary}

\begin{proof}
The exact expression for $\prob (T_x > T_y) = \prob ( T_1 > T_{y/x}
)$ can be derived from Proposition \ref{HittingProb} by using our
choice of $F$ to compute the barycentric coordinate of the $F(0)$
vertex. For \eqref{AsympExpression}, note that $v := y/x \in (0,1)$
and $F$ is a decreasing function on $[0,1]$ with $F(0) = 1$ and
$F(1) = 0$. Therefore it is enough to show that $F(v) \asymp
(1-v)^{4a-1}$ for $v$ slightly less than 1, which follows easily
from \eqref{ExplicitF}. Combining the exact and approximate
expressions with the Domain Markov Property (that is, mapping back
to the upper half plane via $g_\tau$) proves the last statement.
\end{proof}

We get \eqref{upperbound} as an immediate result of Corollary
\ref{RealLineHittingProb}, since
\begin{align*}
\sum_{k=1}^{2^n} \prob \left( D_k^n \right) &\asymp \sum_{k=1}^{2^n}
\left( \frac{1}{k} \right)^{4a-1} \\
&= 2^{(2-4a)n} \sum_{k=1}^{2^n} \left( \frac{1}{k2^{-n}}
\right)^{4a-1} 2^{-n}.
\end{align*}
The summation term is a Riemann sum for $\int_0^1 u^{1-4a} du$,
which is finite for $1/4 < a <  1/2$. This completes the proof of
the upper bound estimate. The next two results will only be used in
Section \ref{LowerBoundSection} but we mention them here as they are
direct corollaries of Proposition \ref{HittingProb}.

\begin{corollary}
\label{trilinear} There are fixed constants $D_0, D_1,$ and
$D_\infty$, depending only on $a$, for which the three swallowing
probabilities of Proposition \ref{HittingProb} satisfy
\begin{align*}
& \prob \left( T_z < T_1 \right) = D_0 \dist( F(z), S_0 ), \\ &\prob
\left( T_z = T_1 \right) = D_1 \dist( F(z), S_1 ), \\ &\prob \left(
T_z > T_1 \right) = D_{\infty} \dist( F(z), S_\infty ),
\end{align*}
where $S_0, S_1,$ and $S_\infty$ are the lines that form the sides
of $T$, opposite the vertices $F(0), F(1)$, and $F(\infty)$,
respectively.
\end{corollary}

\begin{proof}
The statement is an example of the relationship between barycentric
coordinates and \textit{trilinear coordinates}, which describe the
point $F(z)$ using the distances to the three sides of the triangle.
The relationship is clear: the distance from $c_0 F(0) + c_1 F(1) +
c_\infty F(\infty)$ to the line through $F(0)$ and $F(1)$ is a
linear function of $c_\infty$ (and similarly the distances to the
other lines are linear functions of $c_0$ and $c_1$).
\end{proof}

\begin{corollary}
\label{AboveLineHittingProb} For $0 < y < x$, $0 \leq \theta \leq
\pi$, and $r \leq (x-y)/4$,
\begin{align}
\prob \left( T_{x+re^{i\theta}} < T_y \right) \asymp
\frac{y^{1-2a}}{x^{2a}} (x-y)^{4a-2} r \sin \theta.
\end{align}
\end{corollary}

\begin{proof}
Let $z' = (x+re^{i\theta})/y$. By scaling and Corollary
\ref{trilinear},
\begin{align*}
\prob \left( T_{x+re^{i\theta}} < T_y \right) = \prob \left( T_{z'}
< T_1 \right) = D_0 \dist \left( F(z'), S_0 \right).
\end{align*}
A useful tool for estimating a distance to the boundary of a domain
is the Koebe $1/4$ Theorem (see \cite[Corollary 3.19]{lawler:book}),
which states that if $f : D \to D'$ is conformal and $z \in D$ then
\begin{align*}
\frac{\dist \left( f(z), D' \right)}{\dist \left( z, D \right)}
\asymp |f'(z)|,
\end{align*}
where the left and right hand constants implicit in $\asymp$ are
$1/4$ and $4$, respectively. We claim that the conditions $0 < y <
x$ and $r \leq (x-y)/4$ are enough so that $F(z')$ is closest to
side $S_0$ in $T$. Assuming this, it follows that
\begin{align*}
\dist \left( F(z'), S_0 \right) \asymp |F'(z')| \dist \left( z',
\partial \mathbb{H} \right) \propto |z'|^{-2a} |z'-1|^{4a-2}
\textrm{Im}(z').
\end{align*}
Using that $r \leq (x-y)/4$, we have $|z'| \asymp x/y$ and $|z'-1|
\asymp (x/y - 1)$. Clearly $\textrm{Im}(z') = r \sin \theta/y$, from
which the result follows.

\begin{figure}
\begin{center}
\includegraphics[scale=1]{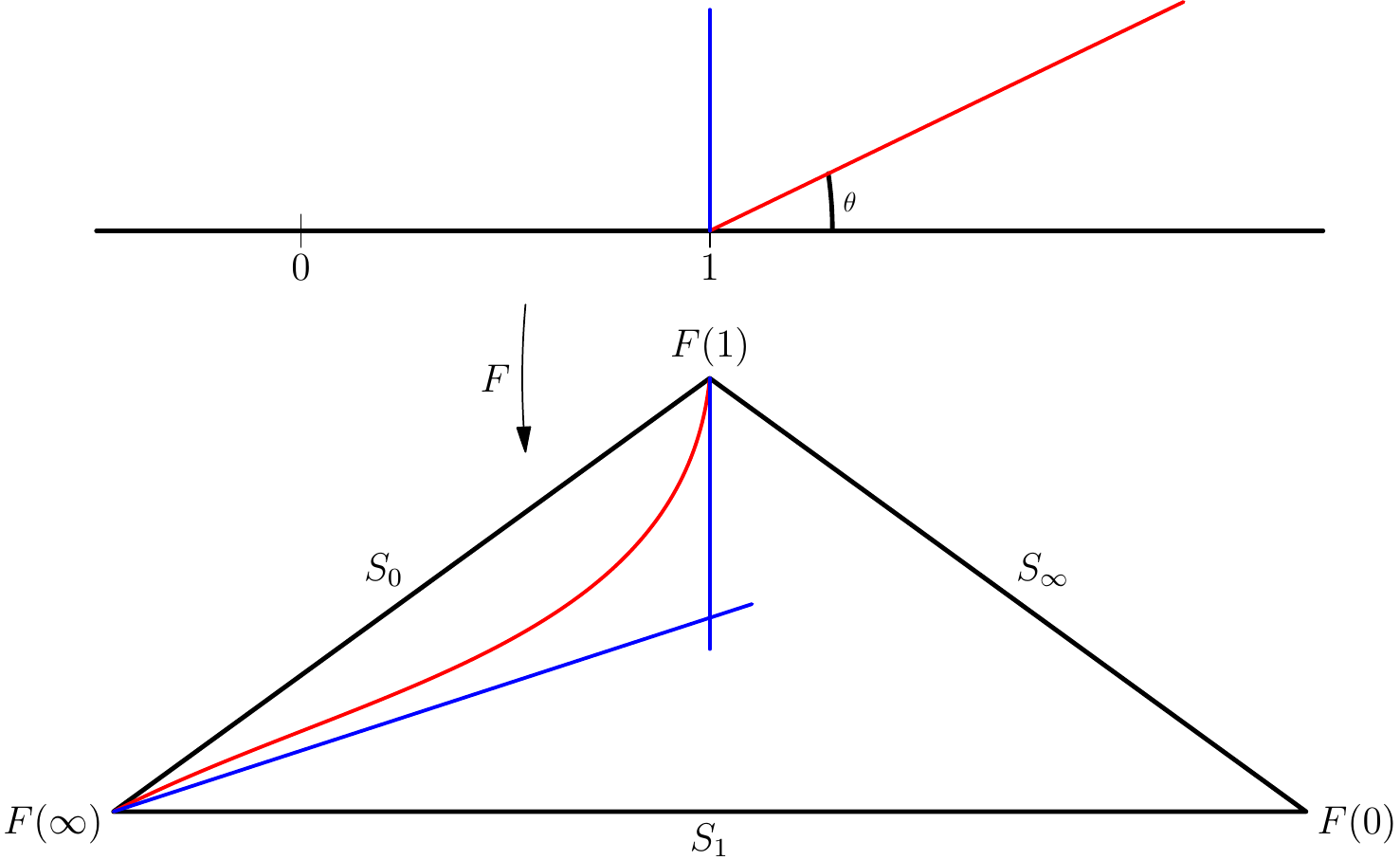}
\caption{The image of the sector $0 \leq \arg(z-1) \leq \theta <
\pi/2$ is, among the three sides of the triangle, always closest to
side $S_0$. This is seen by noting that, in the upper-half plane,
the sector begins on the side of the angle bisector at $F(1)$ that
is closest to $S_0$, and then a curvature argument shows that the
image of the sector must be curving \textit{away} from the angle
bisector. A similar argument shows the curve lies to the left of the
image bisector at $F(\infty)$.} \label{triangle-fig2}
\end{center}
\end{figure}

Now we justify the claim that $F(z')$ is closest to the side $S_0$
in $T$. Let $\alpha \in [0, \pi/2)$. We will show that the curve
$\phi(t) := F(1 + te^{i \alpha})$ lies inside the subtriangle $T'$
bounded by $S_0$ and the two angle bisectors at the vertices $F(1)$
and $F(\infty)$, which proves that it is closest to $S_0$ in $T$. In
the upper half plane the pre-image of the bisector at $F(1)$ is
locally the vertical line from $1$ to $\infty$, and the line $1 +
te^{i\alpha}$ is to the right of this (and closer to the pre-image
of $S_0$, see Figure \ref{triangle-fig2}). Therefore $\phi(t)$ is in
the subtriangle $T'$ for $t$ small at least. But using $F'(z)
\propto z^{-2a}(1-z)^{4a-2}$ it is easy to verify that
\begin{align*}
\partial_t \arg \gamma'(t) = -2a \, \partial_t \arg \left( 1+te^{i\alpha}
\right) \leq 0,
\end{align*}
so that $\phi(t)$ must be curving \textit{away} from the angle
bisector at $F(1)$. Hence $\phi[0, \infty)$ lies on the side of the
bisector closest to $S_0$. A similar argument shows that $\phi[0,
\infty)$ also lies on the side of the angle bisector at $F(\infty)$
that is closest to $S_0$. Since $\textrm{Re}(z') > 1$, we have $z' =
1 + te^{i\alpha}$ for some $t > 0$ and $\alpha \in [0, \pi/2)$,
which proves the claim.
\end{proof}

The constraint $r \leq (x-y)/4$ was not crucial for the above
estimates and certainly could have been improved, but it is all we
will require for later use.

\section{The Two-Interval Estimate \label{LowerBoundSection}}

In this section we work towards establishing the estimates for Lemma
\ref{LowerBoundLemma}. We already get \eqref{lowerbound1} for free
from Corollary \ref{RealLineHittingProb} since
\begin{align*}
\prob \left( D_k^n \right) \asymp k^{1-4a} \geq 2^{(1-4a)n},
\end{align*}
by $k \leq 2^n$. To prove the much more difficult bound
\eqref{lowerbound2} we require an estimate on the SLE curve hitting
two small disjoint intervals. We use various tools from the theory
of conformal mapping to accomplish this.

The case of adjacent intervals, corresponding to $k = j+1$ in
\eqref{lowerbound2}, we will handle directly. In fact in this case
the desired probability can be computed exactly, as the following
Lemma shows.

\begin{lemma}
\label{ExactTwoLemma} Let $0 < x_1 < x_2 < x_3$ be real numbers.
Then
\begin{align*}
\prob \left( T_{x_1} < T_{x_2} < T_{x_3} \right) = \prob \left(
T_{x_1} < T_{x_2} \right) + \prob \left( T_{x_2} < T_{x_3} \right) -
\prob \left( T_{x_1} < T_{x_3} \right).
\end{align*}
\end{lemma}

\begin{proof}
The curve hitting in either interval $[x_1, x_2]$ or $[x_2, x_3]$ is
equivalent to it hitting in $[x_1, x_3]$, from which the result
follows.
\end{proof}
From Lemma \ref{ExactTwoLemma}, the assumption $k2^{-n} > \delta$,
and the approximation in \eqref{AsympExpression}, we have the
existence of a constant $C$ such that
\begin{align*}
\prob \left( D_k^n \cap D_{k+1}^n \right) &\leq C \left( \left(
\frac{2^{-n}}{k2^{-n}} \right)^{4a-1} + \left(
\frac{2^{-n}}{(k+1)2^{-n}} \right)^{4a-1} - \left( \frac{2 \cdot
2^{-n}}{(k+1)2^{-n}} \right)^{4a-1} \right) \\
&\leq \left( \frac{1}{\delta} \right)^{4a-1} (2-2^{4a-1})
2^{-(4a-1)n} \\
&= C_{\delta} 2^{-(4a-1)n}.
\end{align*}
This is exactly \eqref{lowerbound2} for $k-j=1$.

The rest of this section deals with $k-j \geq 2$. It is actually
easier to discuss our proof of \eqref{lowerbound2} if we use a
notation involving continuous variables rather than discrete, so
assume the two intervals are $(y,y+\epsilon)$ and $(x,x+\epsilon)$
with $0 < \delta < y < x < 1 - \delta$ and $\epsilon > 0$.
Implicitly though we mean $x = k2^{-n}, y = j2^{-n},$ and
$\epsilon=2^{-n}$. In this notation, proving \eqref{lowerbound2} is
the same as showing that
\begin{align}
\prob \left( T_y < T_{y+\epsilon}, T_x < T_{x+\epsilon} \right) \leq
C \frac{\epsilon^{2(4a-1)}}{(x-y)^{4a-1}}.
\label{ModifiedLowerBound}
\end{align}
Since we are now assuming that $k-j \geq 2$, we have that $x-y =
(k-j)2^{-n} \geq 2\epsilon$. The bound $\epsilon \leq (x-y)/2$ will
be used later on.

We make a brief note about constants here. In moving from line to
line we do not always explicitly indicate when the constants
involved in a bound may change, usually preferring to fold the new
constants into the generic value $C$. It is important to note that,
in accordance with Lemma \ref{LowerBoundLemma}, any new constants
depend only on $a$ and $\delta$ and never $x, y,$ or $\epsilon$.

For the two-interval hitting probability we already know the
probability of the curve hitting the first interval $(y,
y+\epsilon)$, so we are clearly interested in the conditional
probability of hitting the second interval $(x, x+\epsilon)$ at the
time $y$ is swallowed. Therefore we condition on $\mathcal{F}_{T_y}$
and arrive at
\begin{align}
\prob \left( T_y < T_{y+\epsilon}, T_x < T_{x + \epsilon} \right) &
= \expect \left[ \indicate{T_y < T_{y+\epsilon}} \expect \left[
\indicate{T_x < T_{x+\epsilon}} \mid \mathcal{F}_{T_y} \right]
\right] \notag \\
&\asymp \expect \left[ \indicate{T_y < T_{y+\epsilon}} \left(
\frac{g_{T_y}(x+\epsilon) - g_{T_y}(x)}{g_{T_y}(x+\epsilon) -
B_{T_y}} \right)^{4a-1} \right], \label{SecondMomentLeaveOff}
\end{align}
the last expression being a result of Corollary
\ref{RealLineHittingProb}. This reduces the two-interval hitting
probability to computing a certain moment, but only on the event
$\{T_y < T_{y + \epsilon}\}$ rather than the full space.  Needless
to say this is a complicated calculation. Moreover, it is not a
priori clear how the estimate \eqref{SecondMomentLeaveOff} is
related to the desired bound \eqref{ModifiedLowerBound}. The
following two Lemmas provide the link. We note here that these
Lemmas are deterministic in nature and apply to {\it any} continuous
driving function $U_t$.

\begin{lemma}
\label{Dist2CurveLemma} Suppose that $U_t$ is the driving function
for the Loewner equation. Fix a point $x > 0$, and let $d_t(x) =
\textrm{dist}(x, \partial K_t)$. Define $s_t = \sup K_t \cap
\mathbb{R}$, and let $\eta_t := g_t(s_t+) := \lim_{x \downarrow s_t}
g_t(x)$. Then for $t < T_x$,
\begin{align*}
\frac{g_t(x) - \eta_t}{4 g_t'(x)} \leq d_t(x) \leq 4 \frac{g_t(x) -
\eta_t}{g_t'(x)}.
\end{align*}
In particular, if $T_y < T_x$, then
\begin{align*}
\frac{g_{T_y}(x) - U_{T_y}}{4g'_{T_y}(x)} \leq d_{T_y}(x) \leq
4\frac{g_{T_y}(x) - U_{T_y}}{g_{T_y}'(x)}.
\end{align*}
\end{lemma}

\begin{proof}
Let $\tilde{K}_t$ be the reflection of the hull $K_t$ across the
real axis. Using the Schwarz reflection principle, the map $g_t$ can
be analytically extended as a map on $\mathbb{C} \backslash (K_t
\cup \tilde{K}_t)$, which we then restrict to $\mathbb{C} \backslash
(K_t \cup \tilde{K}_t \cup (-\infty, 0])$ so the domain is simply
connected. The image of the extended $g_t$ is $\mathbb{C}\backslash
(-\infty, \eta_t]$. Noting that $d_t(x) = \textrm{dist}(x,
\partial(K_t \cup \tilde{K}_t))$ by symmetry, a direct application
of the Koebe 1/4 Theorem gives that
\begin{align*}
\frac{D_t(x)}{4 d_t(x)} \leq g_t'(x) \leq \frac{4D_t(x)}{d_t(x)}
\end{align*}
where $D_t(x) = \textrm{dist}(g_t(x), (-\infty, \eta_t]) = g_t(x) -
\eta_t$. This gives the first statement, and for the special case
one only has to note that $\eta_{T_y} = U_{T_y}$ since the tip of
the SLE curve is on the positive real line at time $T_y$.
\end{proof}

\begin{lemma}
\label{RatioEstimateLemma} Let $U_t, x,$ and $d_t(x)$ be as in Lemma
\ref{Dist2CurveLemma}. Then
\begin{align*}
\frac{g_{T_y}(x+\epsilon) - g_{T_y}(x)}{g_{T_y}(x+\epsilon) -
U_{T_y}} \leq 4 \frac{\epsilon}{d_{T_y}(x)}.
\end{align*}
Moreover, if $d_{T_y}(x) > 4\epsilon$, then
\begin{align*}
\frac{g_{T_y}(x+\epsilon) - g_{T_y}(x)}{g_{T_y}(x+\epsilon) -
U_{T_y}} \asymp \frac{\epsilon}{d_{T_y}(x)}.
\end{align*}
\end{lemma}

\begin{proof}
Since $U_{T_y} \leq g_{T_y}(x) \leq g_{T_y}(x+\epsilon)$, we have
\begin{align*}
\frac{g_{T_y}(x+\epsilon) - g_{T_y}(x)}{g_{T_y}(x+\epsilon) -
U_{T_y}} \leq 1,
\end{align*}
and hence the claim is trivial if $d_{T_y}(x) \leq 4\epsilon$. In
the case $d_{T_y}(x) > 4\epsilon$ note that
\begin{align}
\frac{g_{T_y}(x+\epsilon) - U_{T_y}}{g_{T_y}(x+\epsilon) -
g_{T_y}(x)} &= 1 + \frac{g_{T_y}(x) - U_{T_y}}{g_{T_y}(x+\epsilon) -
g_{T_y}(x)} \label{RatioEstimateRef}
\end{align}
and by Lemma 3.2,
\begin{align}
\frac{g_{T_y}(x) - U_{T_y}}{g_{T_y}(x+\epsilon) - g_{T_y}(x)} \asymp
\frac{d_{T_y}(x) g_{T_y}'(x)}{g_{T_y}(x+\epsilon) - g_{T_y}(x)},
\label{RatioEstimateRef2}
\end{align}
where that the left and right constants implicit in $\asymp$ are
$1/4$ and $4$, respectively. The last term can be approximated using
the Growth Theorem (see \cite[Theorem 3.23]{lawler:book}), which
says that if $f : \{ |z| < 1 \} \to \mathbb{C}$ with $f(0) = 0$ and
$f'(0) = 1$ then
\begin{align*}
\frac{|z|}{(1+|z|)^2} \leq |f(z)| \leq \frac{|z|}{(1-|z|)^2}.
\end{align*}
The map
\begin{align*}
\tilde{g}_t(z) = \frac{g_t(z_0 + d_t(z_0)z) - g_t(z_0)}{d_t(z_0)
g'_t(z_0)}
\end{align*}
satisfies these conditions, where $g_t$ is extended onto $\mathbb{C}
\backslash (K_t \cup \tilde{K}_t \cup (-\infty, 0])$ as in Lemma
\ref{Dist2CurveLemma}. Setting $z_0 = x, t = T_y, z = \epsilon/
d_{T_y}(x)$, and using the assumption that $4\epsilon < d_{T_y}(x)$
gives
\begin{align*}
\frac{(1 - \epsilon/d_{T_y}(x))^2}{\epsilon/d_{T_y}(x)} \leq
\frac{d_{T_y}(x) g_{T_y}'(x)}{g_{T_y}(x+\epsilon) - g_{T_y}(x)} \leq
\frac{(1 + \epsilon/d_{T_y}(x))^2}{\epsilon/d_{T_y}(x)}.
\end{align*}
Combining this with \eqref{RatioEstimateRef} and
\eqref{RatioEstimateRef2} we have
\begin{align*}
1 + \frac{(1 - \epsilon/d_{T_y}(x))^2}{4\epsilon/d_{T_y}(x)} \leq
\frac{g_{T_y}(x+\epsilon) - U_{T_y}}{g_{T_y}(x+\epsilon) -
g_{T_y}(x)} \leq 1 + 4 \frac{(1 +
\epsilon/d_{T_y}(x))^2}{\epsilon/d_{T_y}(x)},
\end{align*}
or, what is equivalent,
\begin{align*}
\frac{\epsilon/d_{T_y}(x)}{(1+\epsilon/d_{T_y}(x))^2 + 4
\epsilon/d_{T_y}(x)} \leq \frac{g_{T_y}(x + \epsilon) -
g_{T_y}(x)}{g_{T_y}(x+\epsilon) - U_{T_y}} \leq \frac{4 \epsilon /
d_{T_y}(x)}{(1+\epsilon/d_{T_y}(x))^2}.
\end{align*}
Maximizing (minimizing) the denominator of the left (right) hand
side produces
\begin{align*}
\frac{16}{41} \frac{\epsilon}{d_{T_y}(x)} \leq \frac{g_{T_y}(x +
\epsilon) - g_{T_y}(x)}{g_{T_y}(x+\epsilon) - U_{T_y}} \leq 4
\frac{\epsilon}{d_{T_y}(x)}.
\end{align*}
\end{proof}

With Lemma \ref{RatioEstimateLemma} in hand the relation between
\eqref{ModifiedLowerBound} and \eqref{SecondMomentLeaveOff} becomes
more evident. By \eqref{SecondMomentLeaveOff} and Lemma
\ref{RatioEstimateLemma},
\begin{align}
\prob \left( T_y < T_{y+\epsilon}, T_x < T_{x+\epsilon} \right) \leq
C \epsilon^{4a-1} \expect \left[ \indicate{T_y < T_{y+\epsilon}}
d_{T_y}(x)^{1-4a} \right]. \label{UpperBoundOne}
\end{align}
On the event $\{ T_y < T_{y+\epsilon} \}$ it is important to note
that $d_{T_y}(x)$ satisfies $0 \leq d_{T_y}(x) \leq x-y$. The upper
bound comes from the simple observation that $\gamma(T_y)$ lies
somewhere on the real line to the right of $y$. In fact, on $\{ T_y
< T_{y+\epsilon} \}$ it is even true that $\gamma(T_y) \in [y,
y+\epsilon]$. The latter suggests that $d_{T_y}(x)$ should not be
much less than $x-y$ either, since otherwise the SLE curve would
have to touch somewhere on the real line before $y$, and then make
an excursion in the upper half-plane that gets very close to $x$ but
then returns all the way back to the interval $[y,y+\epsilon]$. One
expects such excursions to be rare. \textit{If} it is true that
$d_{T_y}(x)$ is roughly on the order of $x-y$, then
\eqref{UpperBoundOne} gives
\begin{align*}
\prob \left( T_y < T_{y+\epsilon}, T_x < T_{x+\epsilon} \right)
&\leq C \prob \left( T_y < T_{y+\epsilon} \right)  \epsilon^{4a-1}
(x-y)^{1-4a} \\
&\leq C \left( \frac{\epsilon}{y+\epsilon} \right)^{4a-1}
\epsilon^{4a-1} (x-y)^{1-4a} \\
&\leq C_{\delta} \epsilon^{2(4a-1)}(x-y)^{1-4a},
\end{align*}
where the last inequality uses $y > \delta$. This is exactly
\eqref{ModifiedLowerBound}. The rest of the paper proceeds with this
line of attack in mind, and the crux of the remaining argument is
showing that $d_{T_y}(x)$ is rarely small on the event $\{ T_y <
T_{y+\epsilon} \}$.

Consider the distribution function
\begin{align*}
G(r) = \prob \left( T_y < T_{y+\epsilon}, d_{T_y}(x) \leq r \right).
\end{align*}
We use $G$ to write the expectation in \eqref{UpperBoundOne} as
\begin{align}
\expect \left[ \indicate{T_y < T_{y+\epsilon}} d_{T_y}(x)^{1-4a}
\right]
&= \int_0^{x-y} r^{1-4a} dG(r) \notag \\
&= \int_0^{x-y} \int_r^{\infty} (4a-1)
v^{-4a} dv \, dG(r) \notag \\
&= \int_0^{x-y} (4a-1) v^{-4a} G(v) dv + \int_{x-y}^{\infty} (4a-1)
v^{-4a} G(x-y) dv, \label{GeometricSum}
\end{align}
the last equality being an application of Fubini's Theorem. Consider
the second integral first. For it we have
\begin{align*}
G(x-y) = \prob \left( T_y < T_{y+\epsilon} \right) \asymp \left(
\frac{\epsilon}{y+\epsilon} \right)^{4a-1} \leq C_{\delta}
\epsilon^{4a-1},
\end{align*}
and again the last inequality uses $y > \delta$. Consequently
\begin{align}
\int_{x-y}^{\infty} (4a-1) v^{-4a} G(x-y) dv \leq C
\frac{\epsilon^{4a-1}}{(x-y)^{4a-1}} \label{IntegralTwoBound}
\end{align}
for some constant $C$ depending only on $a$ and $\delta$.

We need the same upper bound for the first integral in
\eqref{GeometricSum}, which requires an upper bound on $G(r)$. By
definition, $G(r)$ is the probability of an SLE curve coming within
a specified distance $r$ of the point $x$ \textit{before} continuing
on to hit the interval $(y,y+\epsilon)$. To estimate $G(r)$ our
strategy will be to decompose any such curve into the path from zero
to where it first hits the semi-circle of radius $r$ centered at
$x$, and then from the semi-circle to the interval $(y, y+\epsilon)$
(see Figure \ref{hull-fig}). The probability of the curve hitting
the semi-circle (before swallowing $y$) will be estimated directly,
and the probability of the curve going from the semi-circle to $(y,
y+\epsilon)$ will be estimated using the conformal invariance
property and some considerations of harmonic measure.

We split the first integral in \eqref{GeometricSum} into two parts:
\begin{align}
\int_0^{x-y} (4a-1)v^{-4a} G(v) dv = \int_0^{\frac{x-y}{4}}
(4a-1)v^{-4a} G(v) dv + \int_{\frac{x-y}{4}}^{x-y} (4a-1)v^{-4a}G(v)
dv. \label{TwoPartFirst}
\end{align}
Using that $G(r)$ is an increasing function of $r$,
\begin{align}
\int_{\frac{x-y}{4}}^{x-y} (4a-1)v^{-4a} G(v) dv &\leq
\int_{\frac{x-y}{4}}^{x-y} (4a-1)\left( \frac{x-y}{4} \right)^{-4a}
G(x-y) dv \notag \\
&\leq C \frac{\epsilon^{4a-1}}{(x-y)^{4a-1}},
\label{IntegralOnePBound}
\end{align}
which is the same upper bound in \eqref{IntegralTwoBound}. For the
integral from zero to $(x-y)/4$ we therefore only need an upper
bound on $G(r)$ for $r$ small, namely $r \leq (x-y)/4$. Again the
condition $r \leq (x-y)/4$ is arbitrary, but it is all we will
require later on.

Now we show how to estimate the probability of the SLE curve going
from the semi-circle to the interval $(y,y+\epsilon)$. Define the
stopping time $\tau_r = \inf \{t \geq 0 : |\gamma(t) - x| = r \}$.
The event $\{d_{T_y}(x) \leq r \}$ is the same as the event $\left\{
\tau_r < T_y \right \}$, and both are clearly
$\mathcal{F}_{\tau_r}$-measurable. We condition on
$\mathcal{F}_{\tau_r}$ to compute the probability of the curve going
from the semi-circle to $(y,y+\epsilon)$, so that
\begin{align}
G(r) = \prob \left( T_y < T_{y+\epsilon}, d_{T_y}(x) \leq r \right)
&\asymp \expect \left[ \indicate{d_{T_y}(x) \leq r} \left(
\frac{g_{\tau_r}(y+\epsilon) - g_{\tau_r}(y)}{g_{\tau_r}(y+\epsilon)
- B_{\tau_r}} \right)^{4a-1} \right] \label{HittingDistance}.
\end{align}

The following lemma gives an upper bound on \eqref{HittingDistance}.
Again we should note that the lemma is essentially deterministic in
nature and holds for any continuous driving function $U_t$.

\begin{lemma}
\label{ExpectationUB} Suppose $\tau_r < T_y$. Then there exists a
constant $C
> 0$, depending only on $a$ and $\delta$, such that
\begin{align}
\frac{g_{\tau_r}(y+\epsilon) - g_{\tau_r}(y)}{g_{\tau_r}(y+\epsilon)
- U_{\tau_r}} \leq C \frac{\epsilon r}{(x-y)^2}.
\label{DeterministicBound}
\end{align}
\end{lemma}

The proof first gives a way of exactly computing the left hand side
of \eqref{DeterministicBound} using the harmonic measure of certain
boundary segments of the hull $\mathbb{H} \backslash K_{\tau_r}$,
and then the upper bound is arrived at by estimating the harmonic
measure terms. Throughout the rest of the paper we let $\beta$
denote a standard complex Brownian motion (independent of the
driving function for the Loewner equation), and for $z \in
\mathbb{C}$ let $\prob_z$ and $\expect_z$ denote probabilities and
expectations for Brownian motion assuming $\beta_0 = z$. Moreover,
given a domain $D \subset \mathbb{C}$ we define $\tau_D = \inf \{ t
\geq 0 : \beta_t \not \in D \}$.

\begin{figure}
\begin{center}
\includegraphics[scale=1]{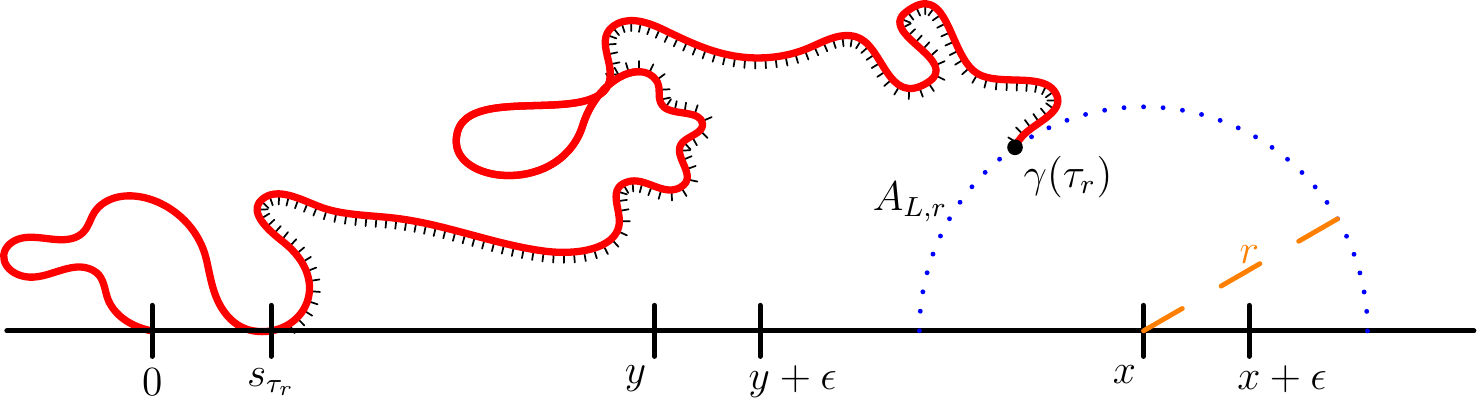}
\caption{The SLE hull at time $\tau_r$. The right hand side of the
hull is highlighted with tick marks.} \label{hull-fig}
\end{center}
\end{figure}

\begin{proof}[Proof of Lemma \ref{ExpectationUB}]
Let $x_1 < x_2$ be real numbers. If $L > 0$, then in the upper
half-plane
\begin{align*}
\prob_{iL} \left( \beta(\tau_{\mathbb{H}}) \in [x_1, x_2] \right) &=
\int_{x_1}^{x_2} \frac{L}{\pi
(x^2 + L^2)} dx \\
&=\frac{x_2 - x_1}{\pi L} +O(L^{-2}),
\end{align*}
which implies
\begin{align*}
x_2 - x_1 = \lim_{L \uparrow \infty} \pi L \cdot \prob_{iL} \left(
\beta(\tau_{\mathbb{H}}) \in [x_1, x_2] \right).
\end{align*}
Consequently,
\begin{align}
\frac{g_{\tau_r}(y+\epsilon) - g_{\tau_r}(y)}{g_{\tau_r}(y+\epsilon)
- U_{\tau_r}} &= \lim_{L \uparrow \infty} \frac{\prob_{iL} \left(
\beta(\tau_\mathbb{H}) \in [g_{\tau_r}(y), g_{\tau_r}(y+\epsilon)]
\right)}{\prob_{iL} \left( \beta(\tau_\mathbb{H}) \in [U_{\tau_r},
g_{\tau_r}(y+\epsilon)] \right)} \label{HarmonicEquivalence}
\end{align}
Using the conformal invariance of Brownian motion, we can compute
the above harmonic measures in the domain $\mathbb{H} \backslash
K_{\tau_r}$ rather than $\mathbb{H}$. Define
\begin{align*}
\mathcal{A}_1 = \{ \beta(\tau_{\mathbb{H} \backslash K_{\tau_r}})
\in [y,y+\epsilon] \}, \,\,\, \mathcal{A}_2 = \{
\beta(\tau_{\mathbb{H} \backslash K_{\tau_r}}) \in [s_{\tau_r},
y+\epsilon] \cup \{ \textrm{right side of } K_{\tau_r} \} \},
\end{align*}
where $s_t$ is as in Lemma \ref{Dist2CurveLemma}. Note $s_{\tau_r} <
y$ since $\tau_r < T_y$. By conformal invariance,
\begin{align*}
\prob_{iL} \left( \beta( \tau_\mathbb{H} ) \in [g_{\tau_r}(y),
g_{\tau_r}(y+\epsilon)] \right) &= \prob_{g_{\tau_r}^{-1}(iL)}(\mathcal{A}_1), \\
\prob_{iL} \left( \beta( \tau_\mathbb{H} ) \in [U_{\tau_r},
g_{\tau_r}(y+\epsilon)] \right) &=
\prob_{g_{\tau_r}^{-1}(iL)}(\mathcal{A}_2).
\end{align*}
Since $g_t$ is normalized so that $g_t(z) = z + o(1)$ as $z \to
\infty$, it follows from \eqref{HarmonicEquivalence} that
\begin{align}
\frac{g_{\tau_r}(y+\epsilon) - g_{\tau_r}(y)}{g_{\tau_r}(y+\epsilon)
- U_{\tau_r}} = \lim_{L \uparrow \infty}
\frac{\prob_{iL}(\mathcal{A}_1)}{\prob_{iL}(\mathcal{A}_2)}.
\label{HittingRatio}
\end{align}
At time $\tau_r$ it is clear that the semi-circle $|z-x|=r$ is
naturally divided into a left arc and a right arc by the point
$\gamma(\tau_r)$ (see Figure \ref{hull-fig}). The left arc we will
refer to as $A_{L,r}$ and the right one as $A_{R,r}$. In the domain
$\mathbb{H} \backslash K_{\tau_r}$ it is clear that the left arc
$A_{L,r}$ naturally ``shields'' the right side of $K_{\tau_r}$ and
the segment $[s_{\tau_r}, y+\epsilon]$, since any Brownian motion
started near infinity that hits these boundaries before any others
must have passed through $A_{L,r}$ first. Hence define the stopping
time
\begin{align*}
\sigma_r = \tau_{\mathbb{H} \backslash K_{\tau_r}} \wedge \inf \{ t
\geq 0 : \beta_t \in A_{L,r} \}.
\end{align*}
Using the Strong Markov Property, the Brownian path from $iL$ to
$[y,y+\epsilon]$ can be decomposed into the path from $iL$ to
$\beta(\sigma_r) \in A_{L,r}$ plus an independent Brownian path from
$\beta(\sigma_r)$ to $[y,y+\epsilon]$. Hence
\begin{align*}
\prob_{iL} \left( \mathcal{A}_1 \right) = \expect_{iL} \left[
\prob_{\beta(\sigma_r)}(\mathcal{A}_1) \right].
\end{align*}
Likewise a similar expression can be derived for the denominator of
\eqref{HittingRatio}, and upon taking the ratio of the two we have
\begin{align*}
\frac{g_{\tau_r}(y+\epsilon) - g_{\tau_r}(y)}{g_{\tau_r}(y+\epsilon)
- U_{\tau_r}} = \lim_{L \uparrow \infty} \frac{\expect_{iL} \left[
\prob_{\beta(\sigma_r)} (\mathcal{A}_1) \right]}{ \expect_{iL}
\left[ \prob_{\beta(\sigma_r)} (\mathcal{A}_2) \right]}.
\end{align*}
Note $\prob_{\beta(\sigma_r)} (\mathcal{A}_1) =
\prob_{\beta(\sigma_r)} (\mathcal{A}_2) = 0$ if $\beta(\sigma_r)
\not \in A_{L,r}$.

Now we take an arbitrary point $z \in A_{L,r}$ and find an upper
bound on $\prob_z(\mathcal{A}_1)$ and a lower bound on
$\prob_z(\mathcal{A}_2)$. The upper bound on
$\prob_z(\mathcal{A}_1)$ is easy, since any Brownian path going from
$z$ to $[y,y+\epsilon]$ in $\mathbb{H} \backslash K_{\tau_r}$ is
also a Brownian path going from $z$ to $[y, y+\epsilon]$ in
$\mathbb{H}$. Hence
\begin{align*}
\pi \prob_z (\mathcal{A}_1) &\leq \pi \prob_z (
\beta(\tau_{\mathbb{H}}) \in [y, y+\epsilon] ) \\
&= \arg(z-y-\epsilon) - \arg(z-y) \\
&= \arg \left( 1 - \frac{\epsilon}{z-y} \right)
\end{align*}
\begin{figure}
\begin{center}
\includegraphics[scale=1]{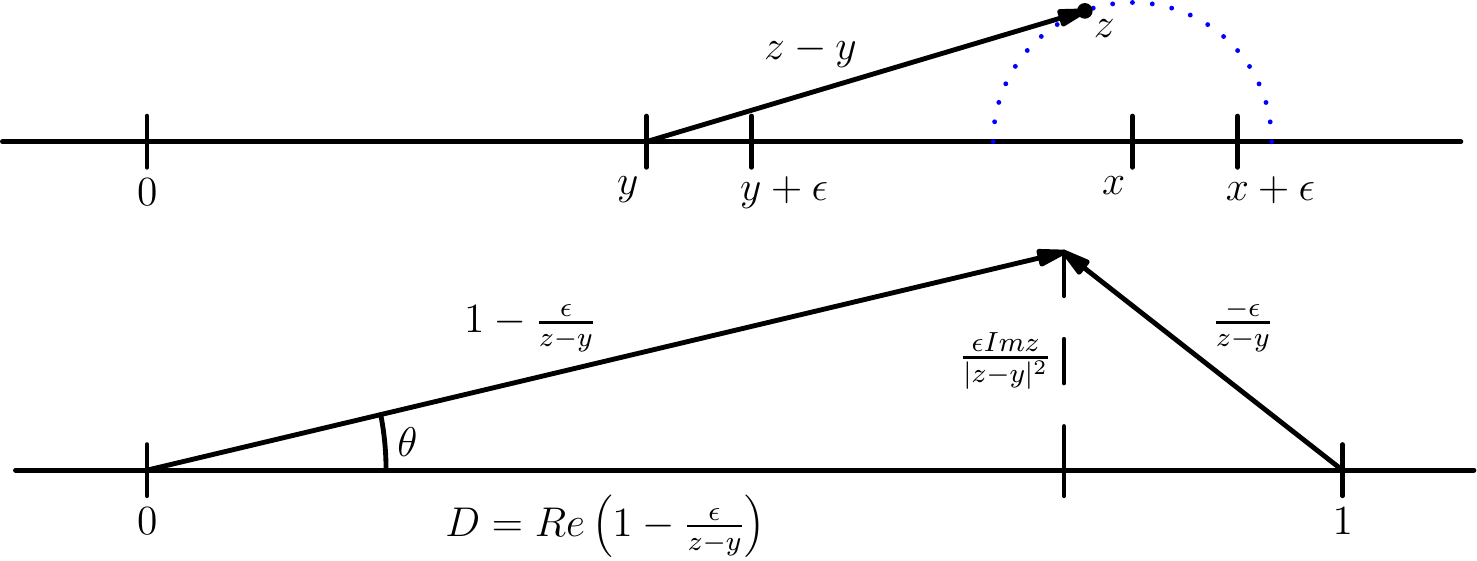}
\caption{Using $r \leq (x-y)/4$ it follows that $|z-y| \geq
\frac{3}{4}(x-y)$. Then by $\epsilon \leq (x-y)/2$ we have
$\frac{\epsilon}{|z-y|} \leq \frac{2}{3}$. Thus $D \geq 1/3$. But
then $\arg \left( 1- \frac{\epsilon}{z-y} \right) = \theta \leq \tan
\theta = \frac{1}{D} \frac{\epsilon \textrm{Im} z}{|z-y|^2} \leq
\frac{16}{3} \frac{\epsilon \textrm{Im} z}{(x-y)^2}$. }
\label{proof-fig}
\end{center}
\end{figure}
Figure \ref{proof-fig} provides a geometric proof, using only
$\epsilon \leq (x-y)/4$ and $r \leq (x-y)/2$, that for some constant
$C > 0$
\begin{align*}
\arg \left( 1 - \frac{\epsilon}{z-y} \right) \leq C \frac{\epsilon
\textrm{Im} z}{(x-y)^2}.
\end{align*}
Hence for all $z \in A_{L,r}$
\begin{align}
\prob_z \left( \mathcal{A}_1 \right) \leq C \frac{\epsilon
\textrm{Im} z}{(x-y)^2}. \label{A1UpperBound}
\end{align}

For $z \in A_{L,r}$ we need a lower bound on
$\prob_z(\mathcal{A}_2)$. Let
\begin{align*}
\mathcal{A}_3 = \mathcal{A}_2 \cap \{ \beta[0, \tau(\mathbb{H}
\backslash K_{\tau_r})] \cap A_{R,r} = \emptyset \}.
\end{align*}
Then $\mathcal{A}_3$ consists of paths in $\mathbb{H} \backslash
K_{\tau_r}$ that exit the domain in $[s_{\tau_r}, y+\epsilon]$ or
the right side of $K_{\tau_r}$ but don't pass through the right arc
$A_{R,r}$ of the semi-circle. Let $V_1 = (-\infty, y+\epsilon) \cup
(x+r, \infty) \cup \{ \textrm{right side of } A_{R,r} \}$, and
\begin{align*}
\mathcal{A}_4 = \{ \beta( \tau(\mathbb{H} \backslash A_{R,r} )) \in
V_1 \}.
\end{align*}
\begin{figure}
\begin{center}
\includegraphics[scale=1]{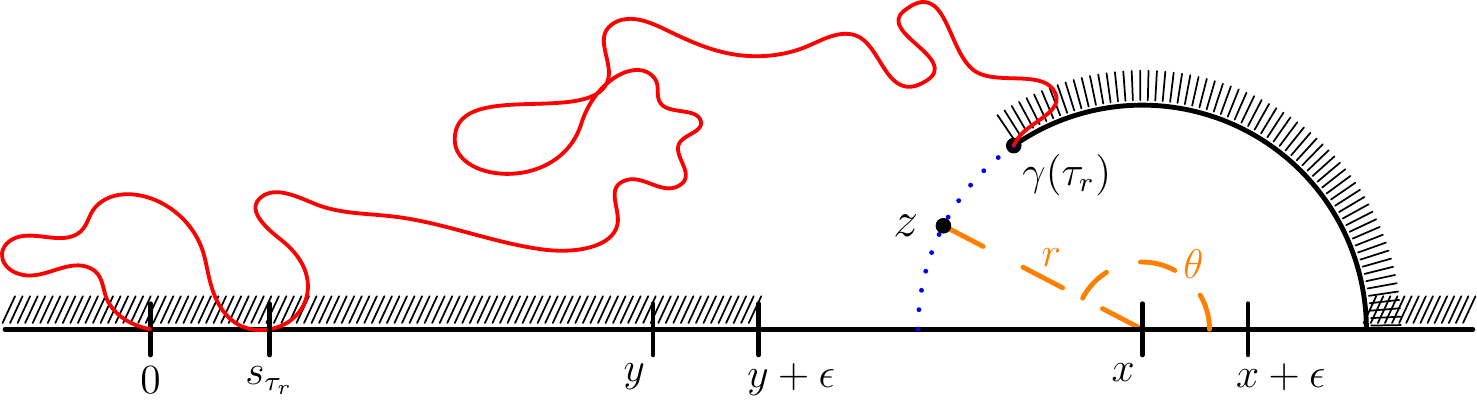}
\caption{The domain $\mathbb{H} \backslash A_{R,r}$ indicated by
solid black boundaries, with the curve $\gamma([0, \tau_r])$ sitting
inside it. The boundary segment $V_1$ is highlighted by tick marks.
Any Brownian path started at $z$ that exits $\mathbb{H} \backslash
A_{R,r}$ on $V_1$ is also a Brownian path in $\mathbb{H} \backslash
K_{\tau_r}$ that exits $\mathbb{H} \backslash K_{\tau_r}$ on
$[s_{\tau_r}, y+\epsilon]$ or the right side of $K_{\tau_r}$.}
\label{arc-fig}
\end{center}
\end{figure}
Topological considerations show that any path in $\mathcal{A}_4$,
started at $z \in A_{L,r}$, must have exited the domain $\mathbb{H}
\backslash K_{\tau_r}$ on $[s_{\tau_r}, y+\epsilon]$ or the right
side of $K_{\tau_r}$ (see Figures \ref{hull-fig} and \ref{arc-fig}),
so that $\mathcal{A}_4 \subset \mathcal{A}_3$. Therefore $\prob_z
(\mathcal{A}_2) \geq \prob_z (\mathcal{A}_3) \geq \prob_z
(\mathcal{A}_4)$. Using basic conformal mappings the probability
$\prob_z (\mathcal{A}_4)$ can be computed explicitly, but for our
purposes a lower bound is sufficient. Map the domain $\mathbb{H}
\backslash A_{R,r}$ into a strip with a slit via $z \mapsto
\log((z-x)/r)$, as shown in Figure \ref{log-fig}(a). Call the image
domain $D$ and let $V_2$ be the image of $V_1$. Let $\theta =
\arg(z-x)$, $\phi = \arg(\gamma(\tau_r) - x)$, so that
\begin{align*}
\prob_z(\mathcal{A}_4) = \prob_{i\theta} \left( \beta(\tau_D) \in
V_2 \right) & \geq \prob_{i\theta} \left( \beta(\tau_D) \in [0,
\infty) \cup \{
\textrm{right side of } [0, i\phi] \} \right) \\
& = \frac{1}{2} \prob_{i \theta} \left( \beta(\tau_D) \in \mathbb{R}
\cup [0, i \phi] \right).
\end{align*}
The last equality is by symmetry. Any Brownian path in the strip $S
= \mathbb{R} \times [0, \pi i]$ that exits $S$ on $\mathbb{R}$ is
also a Brownian path in $D$ that exits $D$ on $\mathbb{R} \cup [0, i
\phi]$, so that
\begin{align*}
\prob_{i \theta} ( \beta(\tau_D) \in \mathbb{R} \cup [0, i \phi])
&\geq \prob_{i \theta} ( \beta(\tau_S) \in \mathbb{R} ) \\
&= \frac{\pi - \theta}{\pi} \\
&\geq \frac{\sin(\pi - \theta)}{\pi} \\
&= \frac{\sin \theta}{\pi} \\
&\geq C \frac{\textrm{Im} z}{r}.
\end{align*}
Therefore there is a constant $C > 0$ such that
\begin{align}
\prob_{z} (\mathcal{A}_2) \geq C \frac{\textrm{Im} z}{r}.
\label{A2LowerBound}
\end{align}
\begin{figure}
\begin{center}
\includegraphics[scale=1]{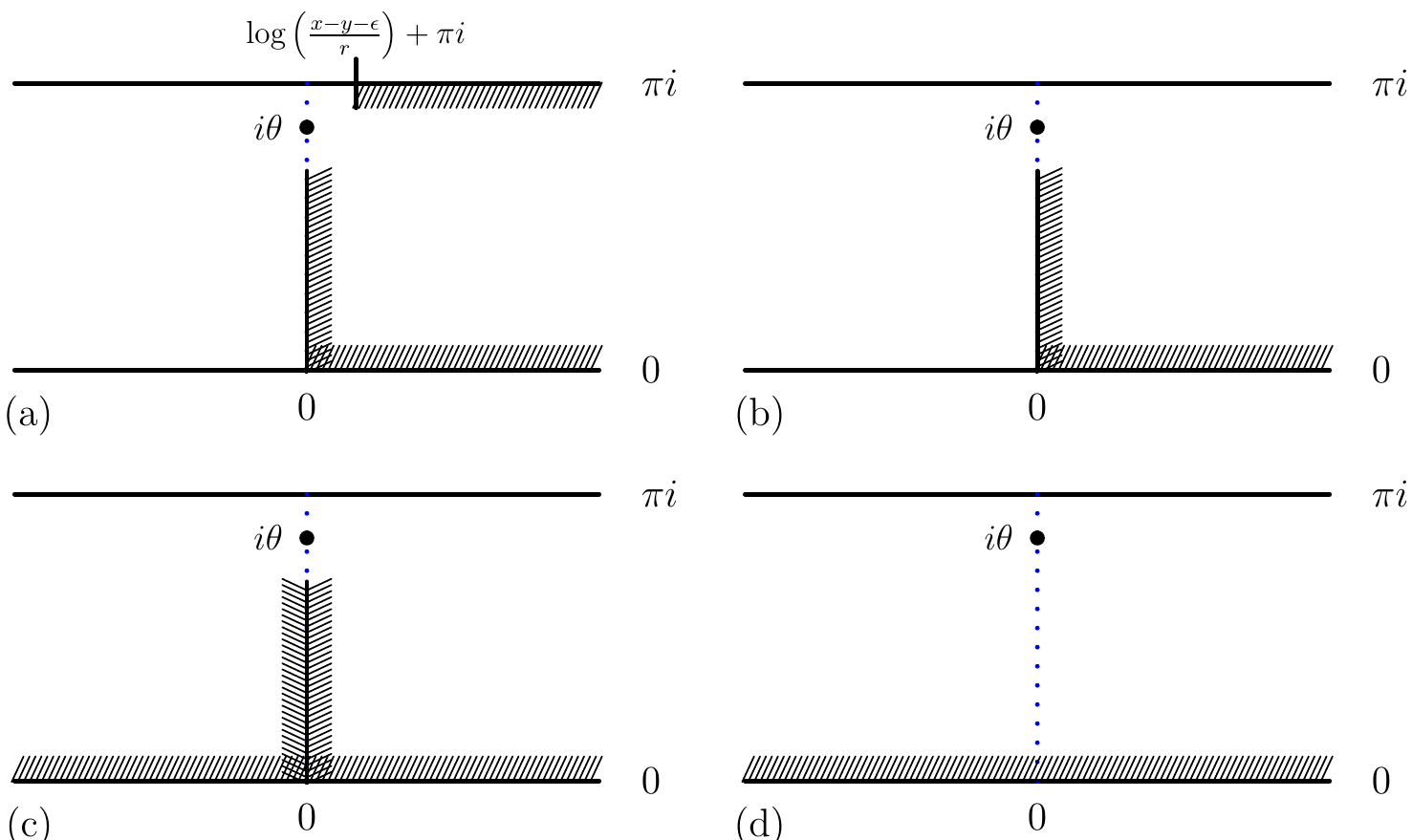}
\caption{(a) The image of the domain $\mathbb{H} \backslash A_{R,r}$
and the point $z$ under the map $w \mapsto \log \left( \frac{w-x}{r}
\right)$. The point $z$ goes to $i \theta$, from which we measure
all the harmonic measure terms. The tick marks highlight the
boundary segment referred to as $V_2$. (b) The harmonic measure of
the highlighted boundary segment is clearly less than the harmonic
measure of $V_2$. (c) The harmonic measure is twice the harmonic
measure in (b), by symmetry. (d) Any Brownian path that exits this
domain on $\mathbb{R}$ must have exited the domain in (c) on
$\mathbb{R}$ or the slit.} \label{log-fig}
\end{center}
\end{figure}

Finally by \eqref{A1UpperBound} and \eqref{A2LowerBound},
\begin{align*}
\prob_{\beta(\sigma_r)}(\mathcal{A}_1) \leq C \frac{\epsilon
\textrm{Im} \beta(\sigma_r)}{(x-y)^2}, \,\,\,
\prob_{\beta(\sigma_r)}(\mathcal{A}_2) \geq C \frac{\textrm{Im}
\beta(\sigma_r)}{r},
\end{align*}
so that
\begin{align*}
\frac{\expect_{iL} \left[ \prob_{\beta(\sigma_r)}(\mathcal{A}_1)
\right]}{\expect_{iL} \left[ \prob_{\beta(\sigma_r)}(\mathcal{A}_2)
\right]} \leq C \frac{\epsilon r}{(x-y)^2}.
\end{align*}
This proves the Lemma.
\end{proof}

Lemma \ref{ExpectationUB} gives us half of the bound on $G(r)$.
Indeed, combining Lemma \ref{ExpectationUB} with
\eqref{HittingDistance} gives
\begin{align} G(r) \leq C \left(
\frac{\epsilon r}{(x-y)^2} \right)^{4a-1} \prob \left( d_{T_y}(x)
\leq r \right). \label{HalfBound}
\end{align}
Now we are only left to estimate the term $\prob (d_{T_y}(x) \leq r)
= \prob (\tau_r < T_y)$. A lower bound is easy, since if the curve
swallows any point on the semi-circle $|z-x| = r$ before $y$ is
swallowed then $\tau_r < T_y$. The probability of $z$ being
swallowed before $y$ is known exactly by Proposition
\ref{HittingProb}, and is well approximated by Corollary
\ref{AboveLineHittingProb}. In fact, choosing $\theta = \pi/2$ in
Corollary \ref{AboveLineHittingProb} gives a lower bound
\begin{align*}
c' \frac{y^{1-2a}}{x^{2a}} (x-y)^{4a-2} r \leq \prob \left( \tau_r <
T_y \right)
\end{align*}
for some constant $c' > 0$. We claim that there is a $C > 0$,
independent of $x,y,$ and $r$, such that
\begin{align}
\prob \left( \tau_r < T_y \right) \leq C \frac{y^{1-2a}}{x^{2a}}
(x-y)^{4a-2} r, \label{SCbeforey}
\end{align}
at least for $r \leq (x-y)/4$. First we suppose that this is true
and show how to get the upper bound estimate
\eqref{ModifiedLowerBound}. From \eqref{SCbeforey} and
\eqref{HalfBound}
\begin{align*}
G(r) \leq C \frac{y^{1-2a}}{x^{2a}} \frac{\epsilon^{4a-1}
r^{4a}}{(x-y)^{4a}} \leq C_{\delta} \frac{\epsilon^{4a-1}
r^{4a}}{(x-y)^{4a}},
\end{align*}
the last inequality coming from $0 < \delta < y < x < 1 - \delta$.
Substituting this into the first integral of \eqref{TwoPartFirst}
gives
\begin{align}
\int_0^{\frac{x-y}{4}} v^{-4a} G(v) dv \leq C
\frac{\epsilon^{4a-1}}{(x-y)^{4a-1}}. \label{IntegralOneBound}
\end{align}
As previously discussed in \eqref{GeometricSum} and
\eqref{TwoPartFirst}, the term $\expect \left[ \indicate{T_y <
T_{y+\epsilon}} d_{T_y}(x)^{1-4a} \right]$ can be broken into three
parts, and then, by \eqref{IntegralTwoBound},
\eqref{IntegralOnePBound}, and \eqref{IntegralOneBound}, each part
is bounded above by $C \epsilon^{4a-1} (x-y)^{1-4a}$. Hence $\expect
\left[ \indicate{T_y < T_{y+\epsilon}} d_{T_y}(x)^{1-4a} \right]
\leq C \epsilon^{4a-1} (x-y)^{1-4a}$, and substituting this into
\eqref{UpperBoundOne} we get that
\begin{align*}
\prob \left( T_y < T_{y+\epsilon}, T_x < T_{x+\epsilon} \right) \leq
C \frac{\epsilon^{2(4a-1)}}{(x-y)^{4a-1}}.
\end{align*}
This last bound is exactly \eqref{ModifiedLowerBound}.

The rest of this section is dedicated to proving \eqref{SCbeforey}.

\begin{lemma}
\label{AllPointHM} Let $w_k = -2^{-k-1} + (1 - 3 \cdot 2^{-k-1})
\frac{\pi}{2}i$ for $k = 1,2,\ldots$, and for $k = -1,-2,\ldots$ let
$w_k = \overline{w_{-k}}$. Let $z_k = x + r \exp \{ w_k +
\frac{\pi}{2} i \}$. Then
\begin{align*}
\prob \left( \bigcup_{|k| \geq 1} T_{z_k} < T_y \right) \leq
\sum_{|k| \geq 1} \prob \left( T_{z_k} < T_y \right) \asymp
\frac{y^{1-2a}}{x^{2a}}(x-y)^{4a-2} r
\end{align*}
\end{lemma}

\begin{proof}
The first inequality is trivial, and using Corollary
\ref{AboveLineHittingProb}
\begin{align*}
\sum_{|k| \geq 1} \prob \left( T_{z_k} < T_y \right) &\asymp
\frac{y^{1-2a}}{x^{2a}} (x-y)^{4a-2} \sum_{|k| \geq 1} r \exp \{
-2^{-|k|-1} \} \sin( \pi - 3 \cdot 2^{-|k|-2} \pi)
\\
& \asymp \frac{y^{1-2a}}{x^{2a}} (x-y)^{4a-2} \sum_{|k| \geq 1} r \sin( 3 \cdot 2^{-|k|-2} \pi) \\
& \asymp \frac{y^{1-2a}}{x^{2a}} (x-y)^{4a-2} r.
\end{align*}
\end{proof}

\begin{figure}
\begin{center}
\includegraphics[scale=.8]{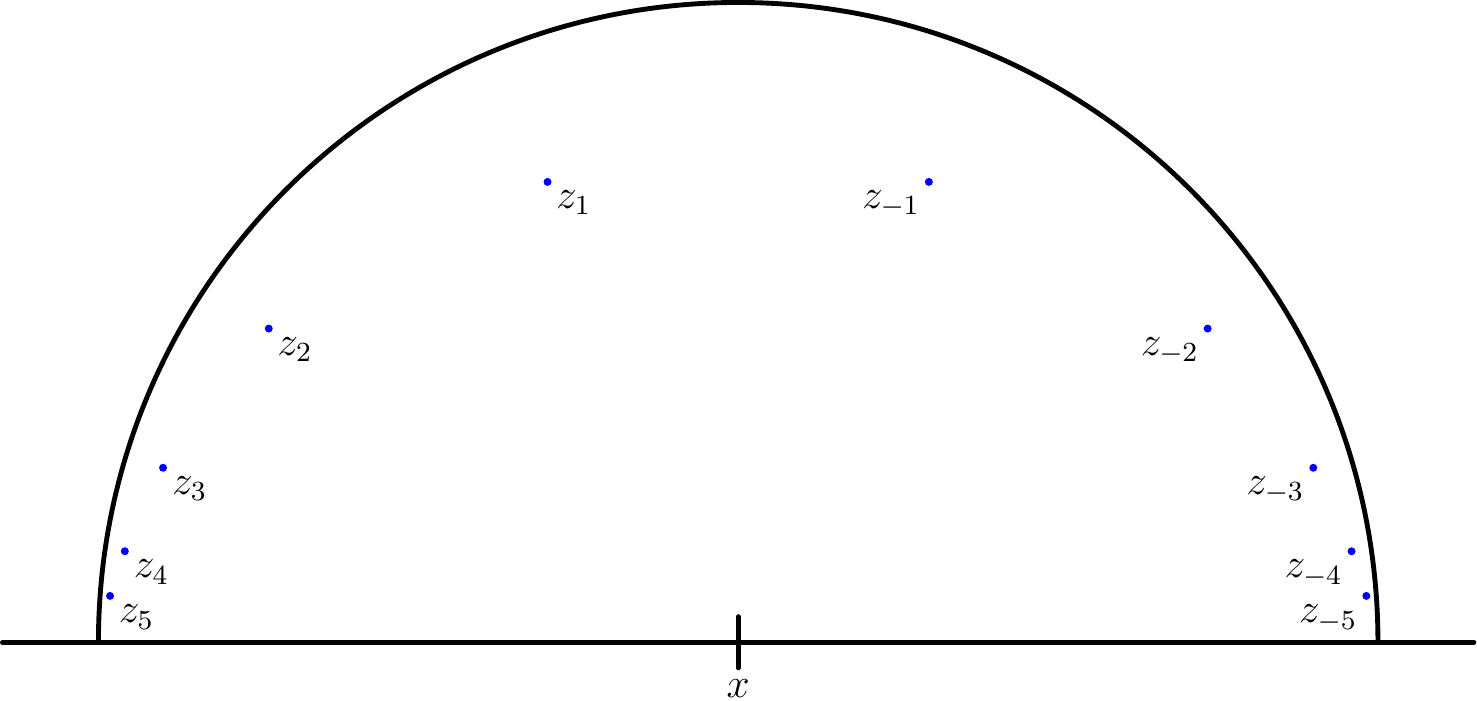}
\caption{The semi-circle of radius $r$ centered at $x$ with the
points $z_k$ inside.} \label{semi-fig}
\end{center}
\end{figure}

Notice that the points $z_k$ sit inside the semi-circle $|z - x| =
r$ (see Figure \ref{semi-fig}), and so if $T_{z_k} < T_y$ for some
$k$ then $\tau_r < T_y$. Conversely, the $z_k$ have been chosen in
such a way that if $\tau_r < T_y$ then it's likely that $T_{z_k} <
T_y$ for some $k$. We prove this last statement shortly, but to do
so we first require a small Lemma on harmonic measure.

\begin{lemma}
\label{RectHM} Let $S$ denote the strip $\mathbb{R} \times [0, \pi
i]$ and let the $w_k$ be as in Lemma \ref{AllPointHM}. There exists
a universal constant $l > 0$ such that if $\phi : [0,1] \to S$ is a
non-self-crossing curve (possibly having multiple points) with
$\textrm{Re } \phi(t) > 0$ for $t \in [0,1)$, $\textrm{Im } \phi(0)
= \pi$, and $\textrm{Re } \phi(1) = 0$ (see Figure
\ref{logracetrack}), and $H$ is the hull that $\phi$ generates (i.e.
the complement of the unbounded connected component of $S \backslash
\phi[0, \infty)$), then $\prob_{w_k} \left( \beta(\tau_{S \backslash
H}) \in \{ \textrm{right side of } \phi \} \right) \geq l$ and
$\prob_{w_k} \left( \beta(\tau_{S \backslash H}) \in \{ \textrm{left
side of } \phi \} \right) \geq l$, for some $k$.
\end{lemma}

\begin{proof}
First consider the sets
\begin{align*}
\mathcal{R}_1 &= \left \{ x+iy : |x| \leq \frac{1}{5} +
\frac{1}{10}, |y| \leq \frac{\pi}{8} +
\frac{1}{10} \right \}, \\
\mathcal{R}_2 &= \left \{ x+iy : |x| \leq \frac{1}{5}, |y| \leq
\frac{\pi}{8} \right \},
\end{align*}
and $\mathcal{R} = \mathcal{R}_1 \backslash \mathcal{R}_2$. A sketch
of $\mathcal{R}$ is given in Figure \ref{racetrack}. Note that $w_0
:= -1/4 \in \mathcal{R}$. Let $\mathcal{L}$ be the line segment from
$-\pi i/8$ to $-\pi i/8 - i/10$, and $\mathcal{L}'$ be the complex
conjugate of the set of points in $\mathcal{L}$. Consider a Brownian
particle started at $w_0$ and killed when it hits the boundary of
$\mathcal{R}$. There is a positive probability that the particle
arrives at $\mathcal{L}$ in the clockwise direction before it
arrives there in the counterclockwise direction, call this
probability $l$. By symmetry this is also the probability that the
particle first reaches $\mathcal{L}'$ in the counterclockwise
direction. An important feature of this probability $l$ is that it
is invariant under scalings and translations of the rectangle
$\mathcal{R}$. We  now cover the imaginary axis from $0$ to $\pi i$
with scaled and translated versions of $\mathcal{R}$ that send $w_0$
to the various $w_k$, as in Figure \ref{logracetrack}. The idea is
that the tip of the curve $\phi(1)$ lies inside one of the
rectangles in Figure \ref{logracetrack}, and then for this rectangle
if the Brownian particle travels from $w_k$ to $\mathcal{L}$ in the
clockwise direction before reaching it in the counterclockwise
direction then it must have hit the right hand side of the curve
$\phi$. The next paragraph provides the details of this argument.
\begin{figure}
\begin{center}
\includegraphics[scale=.8]{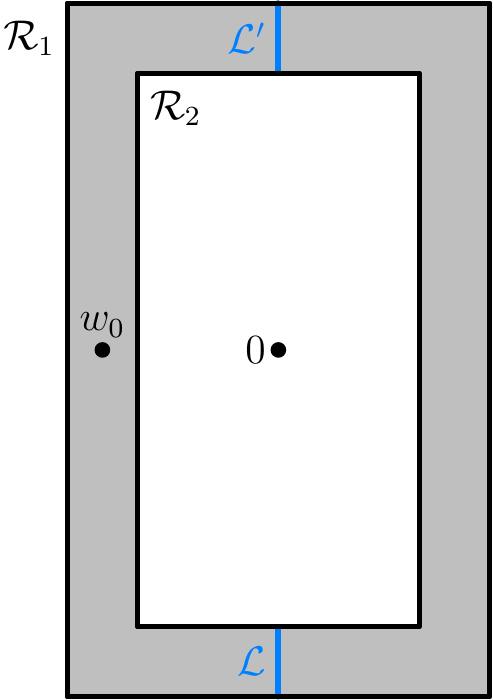}
\caption{The set $\mathcal{R}$ (the shaded region). We let $l$ be
the probability that a Brownian particle started at $w_0$ hits
$\mathcal{L}$ in the clockwise direction before hitting it in the
counterclockwise direction.} \label{racetrack}
\end{center}
\end{figure}

Let $\theta = \textrm{Im } \phi(1) \in [0, \pi]$. Choose the integer
$k$ as follows: if $\theta \geq \pi/2$ then let $k \geq 1$ be such
that $(1 - 2^{-k+1})\pi/2 \leq \theta - \pi/2 \leq (1-2^{-k})\pi/2$,
otherwise let $k \leq -1$ be such that $(1-2^{k+1})\pi/2 \leq \pi/2
- \theta \leq (1-2^k)\pi/2$. Then take the rectangle $\mathcal{R}$
and the point $w_0$, scale them by a factor of $2^{-|k| + 1}$, and
translate both so that the point $w_0$ coincides with point $w_k$.
By construction the point $\phi(1)$ lies somewhere on the vertical
line subdividing the inner rectangle $\mathcal{R}_2$, and the curve
$\phi(t)$ divides the set $\mathcal{R}$. An example with $\theta \in
[\pi/2, 3\pi/4]$ and $k=1$ is shown in Figure \ref{logracetrack}.
For topological reasons, a Brownian particle started at $w_k$ that
hits the line segment $\mathcal{L}$ in the clockwise direction must
have intersected the right side of $\phi$ along the way. This shows
that $\prob_{w_k} \left( \beta( \tau_{S \backslash H} \in \{
\textrm{right side of } \phi \} \right) \geq l$. A completely
symmetrical argument proves the Lemma for the left hand side of
$\phi$.
\begin{figure}
\begin{center}
\includegraphics[scale=1]{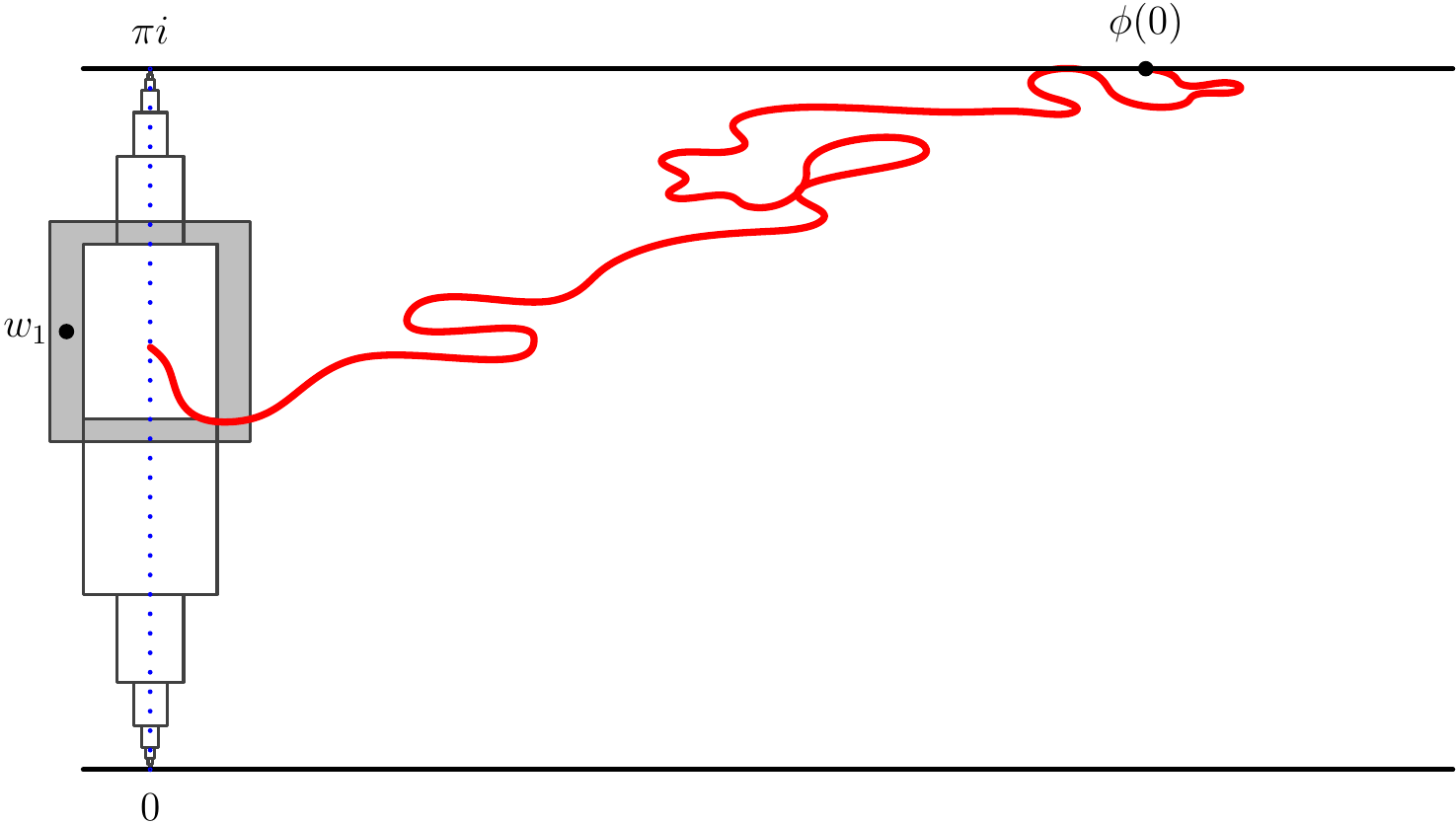}
\caption{The imaginary axis is covered by scaled and shifted
versions of the rectangle $\mathcal{R}_2$. The point $\phi(1)$ must
lie inside one of them, in this case it's the rectangle
corresponding to $k=1$. From the point $w_1$ the harmonic measure of
each side of the curve must be at least $l$.} \label{logracetrack}
\end{center}
\end{figure}
\end{proof}

\begin{lemma}
\label{OnePointHM} Let $z_k$ be as in Lemma \ref{AllPointHM}. There
exists a $c > 0$ such that
\begin{align*}
\prob \left( \left. \bigcup_{|k| \geq 1} T_{z_k} < T_y \right |
\tau_r < T_y \right) \geq c,
\end{align*}
for all $r \leq \frac{x-y}{4}$. The constant $c$ is independent of
$x,y,$ and $r$.
\end{lemma}

\begin{proof}
We will actually prove the stronger statement
\begin{align*}
\prob \left( T_{z_k} < T_y \textrm{ for some k } \mid
\mathcal{F}_{\tau_r} \right) \geq c \indicate{\tau_r < T_y}.
\end{align*}
Let
\begin{align}
\hat{g}_t(z) = \frac{g_t(z) - U_t}{g_t(y) - U_t},
\end{align}
which is well-defined for $t < T_y$, maps from $\mathbb{H}
\backslash K_t \to \mathbb{H}$ and sends $\gamma(t) \to 0, y \to 1,$
and $\infty \to \infty$. Also let $H_t = F \circ \hat{g}_t :
\mathbb{H} \backslash K_t \to T$, where $F$ is the
Schwarz-Christoffel map from Lemma \ref{HittingProb} and $T$ is the
triangle that $F$ maps into. By the Domain Markov Property and
Corollary \ref{trilinear},
\begin{align*}
P \left( T_z < T_y \mid \mathcal{F}_t \right) = D_0 \dist( H_t(z),
S_0), \,\, \textrm{for } t < T_y \wedge T_z.
\end{align*}
Since $|z_k - x| \leq r$ we know $T_{z_k} \geq \tau_r$, so that
\begin{align*}
P \left( T_{z_k} < T_y \mid \mathcal{F}_{\tau_r} \right) = D_0
\dist( H_{\tau_r}(z_k), S_0), \,\, \textrm{for } \tau_r < T_y.
\end{align*}
Clearly then it is enough to find a $c > 0$ such that $\dist(
H_{\tau_r}(z_k), S_0) \geq c$ for some $k$. Again we turn to
harmonic measure estimates. Let $l$ be the universal constant from
Lemma \ref{RectHM} and consider a point $w \in T$ such that a
Brownian particle in $T$, started at $w$, has at least probability
$l$ of hitting the side $S_1$ before any other, and also probability
$l$ of hitting $S_\infty$ before any other side of $T$. Then $w$
cannot be arbitrarily close to $S_0$, otherwise the probability of
hitting one of the sides $S_1$ or $S_\infty$ would have to be small,
so there exists a constant $c = c(l,a)$ such that $\dist(w, S_0)
\geq c$. Hence it is enough to show that for some $k$, a Brownian
particle in $T$, started at $H_{\tau_r}(z_k)$, has at least
probability $l$ of hitting side $S_1$ first, and also probabilty $l$
of hitting side $S_\infty$ first. Using the conformal invariance of
Brownian motion, and noting that the map $H_{\tau_r}^{-1}$
identifies the sides $S_1, S_\infty$ of $T$ with the boundaries $U_1
= (-\infty,0) \cup \{ \textrm{left side of } K_{\tau_r} \} ,
U_\infty = [0,y] \cup \{ \textrm{right side of } K_{\tau_r} \}$ of
$\mathbb{H} \backslash K_{\tau_r}$ (respectively), this is
equivalent to showing a Brownian particle in $\mathbb{H} \backslash
K_{\tau_r}$, started at $z_k$, has probability at least $l$ of
hitting the boundary segment $U_1$ first, and probability at least
$l$ of hitting the boundary segment $U_\infty$ first. But Lemma
\ref{RectHM} already proves this last statement; all that is left to
do is to map $\mathbb{H}$ to the strip $U$ via $z \mapsto \log
((z-x)/r)$ and note that the points $z_k$ go to the points $w_k$.
\end{proof}

Lemmas \ref{OnePointHM} and \ref{AllPointHM} now combine to show
\begin{align*}
\prob ( \tau_r < T_y ) \leq \frac{1}{c} \prob \left( \bigcup_{|k|
\geq 1} T_{z_k} < T_y \right) \leq C' \frac{y^{1-2a}}{x^{2a}}
(x-y)^{4a-2} r.
\end{align*}
This completes the proof of \eqref{SCbeforey}, and also of the
two-interval estimate \eqref{ModifiedLowerBound}.

\bigskip
\bigskip

\noindent \textbf{Acknowledgements:} We thank Greg Lawler for some
very helpful ideas in coming up with the proof presented in Section
\ref{LowerBoundSection}.

\bibliographystyle{alpha}
\bibliography{../SLEBiB/SLE}

\end{document}